\documentclass[draft]{amsart}

\usepackage{amssymb,graphicx,enumerate,amsthm}
\usepackage{multirow}

\newtheorem{theorem}{Theorem}[section]

\renewcommand{\mod}[1]{\text{ (mod }#1)}

\newcommand{\jacobi}[2]{\left( \frac{#1}{#2} \right)}

\numberwithin{equation}{section}

\begin{document}

\makeatletter
\def\imod#1{\allowbreak\mkern10mu({\operator@font mod}\,\,#1)}
\makeatother

\author{Alexander Berkovich}
   	\address{Department of Mathematics, University of Florida, 358 Little Hall, Gainesville FL 32611, USA}
   	\email{alexb@ufl.edu}




\title[\scalebox{.9}{On the $q$-binomial identities involving the Legendre symbol modulo $3$.}]{On the $q$-binomial identities involving the Legendre symbol modulo $3$.}

\begin{abstract} 

I use polynomial analogue of the Jacobi triple product identity together with the Eisenstein formula for the Legendre symbol modulo $3$ 
to prove six identities involving the $q$-binomial coefficients. These identities are then extended to the new infinite hierarchies of $q$-series identities 
by means of the special case of Bailey's lemma. Some of the identities of Ramanujan, Slater, McLaughlin and Sills are obtained this way.
\end{abstract}

\keywords{Eisenstein formula for the Legendre symbol, $q$-binomial identities, infinite hierarchies of $q$-series identities, Bailey's lemma}
  
\subjclass[2010]{Primary 11B65; Secondary 11C08, 11P81, 11P82, 11P83, 11P84,  05A10, 05A15, 05A17}


\date{\today}
   
   
\maketitle

\section{Introduction}\label{Sec:Intro}

\hskip 0.05in

Let $a$ and $q$ be variables and define the $q$-Pochhammer symbol $(a;q)_n :=(1-a)(1-aq)\dots(1-aq^{n-1})$ for any non-negative integer $n$. For $|q|<1$, we define $(a;q)_\infty := \lim_{n\rightarrow\infty} (a;q)_n$. We define the shorthand notation $(a_1,a_2,\dots,a_k;q)_n :=(a_1;q)_n(a_2;q)_n\dots (a_k;q)_n$. Finally note that $1/(q;q)_n = 0$ for all negative $n$.\\

Next, we define the $q$-binomial coefficients
\begin{align*}
{m+n \brack m}_q &:= \left\lbrace \begin{array}{ll}\frac{(q;q)_{m+n}}{(q;q)_m(q;q)_{n}},&\text{for }m, n \geq 0,\\
   0,&\text{otherwise,}\end{array}\right.,
\end{align*}
where $m,n$ are non-negative integers. We would also require the $q$-binomial recurrences [\cite{GR}, I.45, p.353]
\begin{equation}
{m+n \brack m}_q = {m+n-1 \brack m}_q + q^{n} {m+n-1 \brack m-1}_q.
\label{eq:Binom_rec}
\end{equation}
The following two limits are well known.
For any $j\in \mathbb{Z}_{\geq0}$ and $a =0$ or 1,
\begin{align}
\nonumber
 \displaystyle \lim_{L\rightarrow\infty}{L\brack j}_q &= \frac{1}{(q;q)_j},\\
  \nonumber
  \displaystyle \lim_{L\rightarrow\infty}{2L+a\brack L-j}_q &= \frac{1}{(q;q)_\infty}.
\end{align}
In \cite{BU} we proved many $q$-binomial identities involving the Legendre symbol $\mod 3$ 
\begin{align}
\jacobi{j}{3} &= 
\left\lbrace \begin{array}{lll} 
 1,&\text{ if } j\equiv 1 \mod 3,\\
-1,&\text{ if } j\equiv -1 \mod 3,\\  
 0,&\text{ if } j\equiv 0 \mod 3.
\end{array}\right.
\label{}
\end{align}
In particular, we established
\begin{theorem}\label{T1.3}[{Berkovich, Uncu} \cite{BU}] 
Let $L\in \mathbf Z_{\geq 0}$, then
\begin{equation}
\sum_{m,n\geq 0}\frac{q^{2m^2+6mn+6n^2}(q;q)_L}{(q;q)_{L-3n-2m}(q;q)_m(q^3;q^3)_n}=\sum_{j=-L}^L \jacobi{j+1}{3} q^{j^2}{2L\brack L-j}_q.
\label{1.3}
\end{equation}
\end{theorem}
\noindent
In the limit $L\rightarrow\infty$ we have
\begin{equation}
\sum_{m,n\geq 0}\frac{q^{2m^2+6mn+6n^2}}{(q;q)_m(q^3;q^3)_n}=\frac{Q(q^6,-q)}{(q)_\infty},
\label{cpi}
\end{equation}
where
\begin{equation}
Q(q,z):=(q,-z,-\frac{q}{z};q)_\infty(\frac{q}{z^2},z^2q;q^2)_\infty.
\label{capid}
\end{equation}
In deriving \eqref{cpi} we used 
\begin{theorem}\label{T1.6}{Quintuple product identity} [\cite{GR}, ex 5.6,p 147]\\
For $0<|q|<1$ and $z\neq 0$, 
\begin{equation}
\sum_{k=-\infty}^\infty(-1)^k q^{\frac{3k^2-k}{2}}z^{3k}(1+zq^k)=Q(q,z).
\label{qpi}
\end{equation}
\end{theorem}
\noindent
Identity \eqref{cpi} was independently found by Kanade--Russell \cite{KR} and Kur\c{s}ung\"oz \cite{K}.
It represents the analytic version of Capparelli's Theorem \cite{C_thesis}, \cite{C}.\\ 

In this paper we prove many new $q$-binomial identities involving the Legendre symbol $\mod 3$ such as
\begin{equation}
\sum_{j=-\infty}^\infty (-1)^j q^{j^2} \jacobi{j+1}{3} {2L+1\brack L+j}_{q^2}=\frac{(q^3;q^6)_L}{(q;q^2)_{L+1}}(1-q^{2(1+2L)}).
\label{}
\end{equation}
To this end we will employ
\begin{theorem}\label{T1.4}{Polynomial Analogue of the Jacobi Triple Product Identity} [\cite{A1} p. 49]
For $n,m\in \mathbf Z_{\geq 0}$,
\begin{equation}
\sum_{i=-n}^m q^{i^2} x^i{n+m\brack n+i}_{q^2}=(-q/x;q^2)_n(-qx;q^2)_m.
\label{1.4}
\end{equation}
\end{theorem}
\noindent
It is easy to check that \eqref{1.4} is a special case of the $q$-binomial theorem
\begin{equation}
\sum_{i=0}^L q^{i^2} x^i{L\brack i}_{q^2}=(-xq;q^2)_L,
\label{1.4a}
\end{equation}
with $L\in \mathbb{Z}_{\geq 0}$.\\
As $n,m\rightarrow\infty$, \eqref{1.4} becomes the Jacobi triple product identity
\begin{equation}
\sum_{i=-\infty}^\infty q^{i^2} x^i=(-q/x,-qx,q^2;q^2)_\infty.
\label{1.5}
\end{equation}

Eisenstein established the following formula for the Legendre symbol
\begin{equation}
\jacobi{j}{p}=\prod_{n=1}^{\frac{p-1}{2}}\frac{\sin(\frac{2\pi}{p}nj)}{\sin(\frac{2\pi}{p}n)},
\end{equation}
where $p$ is an odd prime.\\
Hence, 
\begin{equation}
\jacobi{j}{3}=\frac{\sin(\frac{2\pi}{3}j)}{\sin(\frac{2\pi}{3})}=\frac{w^j-\bar{w}^j}{w-\bar w},
\label{1.6}
\end{equation}
where $w=exp(\frac{2\pi}{3}I)$, $\bar{w}=w^{-1}$.\\

In Section 2, we will employ both \eqref{1.6} and Theorem~\ref{T1.4}. We would also need
\begin{theorem}\label{T1.1}{Special Case of Bailey's Lemma} \cite{P}, \cite{A}\\
For $a=0,1$, if
\begin{equation}
F_a(L,q)=\sum_{j=-\infty}^\infty\alpha_j(q){2L+a\brack L-j}_q,
\label{}
\end{equation}
then
\begin{equation}
\sum_{r\geq 0}\frac{q^{r^2+ar}(q;q)_{2L+a}}{(q;q)_{L-r}(q;q)_{2r+a}}F_a(r,q)=\sum_{j=-\infty}^\infty\alpha_j(q)q^{j^2+aj}{2L+a\brack L-j}_q.
\label{}
\end{equation}
\end{theorem}
\noindent
Observe that the right-hand side of the second equation in Theorem~\ref{T1.1} is of the same form as the right-hand side of the first equation.
Thus, we may iterate Theorem~\ref{T1.1} as often as we desire by updating $\alpha_j(q)$'s in each step. This procedure gives rise to an infinite hierarchy of polynomial identities.

\section{Six new $q$-hypergeometric hierarchies}\label{Sec:2}

\hskip 0.05in

In this section we will use polynomial analogue of the Jacobi triple identity to prove validity of six "seed" identities. These seeds are then extended to infinite hierarchies of $q$-series identities by means of Theorem~\ref{T1.1}.
\begin{theorem}\label{T1.5}{} 
\begin{equation}
\sum_{j=-\infty}^\infty \jacobi{j}{3} q^{j-1\choose 2} {2L\brack L+j}_q=
\left\{\begin{array}{ll}\frac{(-1;q^3)_{L-1}}{(-1;q)_{L-1}}q^{L-1}\frac{1-q^3}{1-q}(1-q^L), & L>0 \\ 0, & L=0\end{array}\right..
\label{*}
\end{equation}
\end{theorem}
\begin{proof}
\begin{equation}
\begin{split}
&\sum_{j=-\infty}^\infty \jacobi{j}{3} q^{j^2-3j+2} {2L\brack L+j}_{q^2}=
q^2\sum_{j=-\infty}^\infty q^{j^2}\frac{w^j-\bar{w}^j}{w-\bar w}q^{-3j}{2L\brack L+j}_{q^2}=\\
&2\Re(\frac{q^2}{w-\bar w}(-q^4\bar w;q^2)_L(-q^{-2}w;q^2)_L)=\\
&\left\{\begin{array}{ll}(-w,-\bar w;q^2)_{L-1}q^{2(L-1)}(1-q^{2L})(1+q^2+q^4), & L>0 \\ 0, & L=0\end{array}\right.=\\
&\left\{\begin{array}{ll}\frac{(-1;q^6)_{L-1}}{(-1;q^2)_{L-1}}\frac{1-q^6}{1-q^2}q^{2(L-1)}(1-q^{2L}), & L>0 \\ 0, & L=0\end{array}\right..
\end{split}
\label{1.7}
\end{equation}
Replace $q^2\rightarrow q$ to arrive at \eqref{*}.
\end{proof}
\noindent
Applying Theorem~\ref{T1.1} with $a=0$ to \eqref{*} we derive
\begin{equation}
\frac{1-q^3}{1-q}\sum_{n\geq 0} q^{n^2+3n}\frac{(q)_{2L}}{(q)_{2n+2}(q)_{L-1-n}}\frac{(-1;q^3)_n}{(-1;q)_n}(1-q^{n+1})=
\sum_{j=-\infty}^\infty \jacobi{j}{3} q^{3{j\choose 2}}{2L\brack L+j}_q.
\label{}
\end{equation}
Applying Theorem~\ref{T1.1} with $a=0$ repeatedly "$v$-times" to \eqref{*} we derive
\begin{equation}
\begin{split}
\frac{1-q^3}{1-q}&\sum_{n_1,\ldots,n_v\geq 0} \frac{q^{\sum_{i=1}^v (N_i+2)N_i+n_v}} {(q)_{n_1\ldots}(q)_{n_{v-1}}(q)_{2n_v+2}}
\frac{(-1;q^3)_{n_v}}{(-1;q)_{n_v}} (1-q^{1+n_v}) \frac{(q)_{2L}}{(q)_{L-1-N_1}}=\\
&\sum_{j=-\infty}^\infty \jacobi{j}{3} q^{\frac{(2v+1)j^2-3j}{2}-(v-1)}{2L\brack L+j}_q,
\end{split}
\label{l}
\end{equation}
where, here and everywhere, for $j=1,2,\ldots,v$ 
\begin{equation}
N_i=n_i+n_{i+1}+\cdots+n_v.
\label{}
\end{equation}
Multiplying both sides of \eqref{l} by $\frac{1-q}{1-q^3}$ and letting $L\rightarrow\infty$, we have with the aid of \eqref{qpi}
\begin{equation}
\begin{split}
&\sum_{n_1,\ldots,n_v\geq 0} \frac{q^{\sum_{i=1}^v (N_i+2)N_i+n_v}(1-q^{1+n_v})} {(q)_{n_1\ldots}(q)_{n_{v-1}}(q)_{2n_v+2}}
\frac{(-1;q^3)_{n_v}}{(-1;q)_{n_v}}= \frac{1-q}{1-q^3}\\
&\sum_{j=-\infty}^\infty \frac{ \jacobi{j}{3} q^{\frac{(2v+1)j^2-3j}{2}+1-v}}{(q)_\infty}=
\frac{(q^{6v+3},q^{3v},q^{3v+3};q^{6v+3})_\infty(q^{12v+3},q^{12v+9};q^{12v+6})_\infty}{(q^2;q)_\infty}.
\end{split}
\label{}
\end{equation}

Next, we will prove
\begin{theorem}\label{T1.7}{} 
\begin{equation}
\sum_{j=-\infty}^\infty q^{{j-1\choose 2}} \jacobi{j}{3} {2L+1\brack L+j}_q=
q^L\frac{(-1;q^3)_L}{(-1;q)_L}(1+q-q^{L+1}).
\label{**}
\end{equation}
\end{theorem}
\begin{proof}
\begin{equation}
\begin{split}
&\sum_{j=-\infty}^\infty q^{j^2-3j+2} \jacobi{j}{3} {2L+1\brack L+j}_{q^2}=
q^2\sum_{j=-\infty}^\infty q^{j^2}\frac{w^j-\bar{w}^j}{w-\bar w}q^{-3j}{2L+1\brack L+j}_{q^2}=\textit{ by \eqref{1.4}}\\
&2\Re(\frac{1}{\bar w-w}(-\bar w,-w;q^2)_L (1+\bar wq^{2L}(1+q^2)+\bar w^2 q^{2(2L+1)}))=\\
&q^{2L}(-\bar w,-w;q^2)_L \cdot (1+q^2-q^{2L+2})=
 q^{2L}\frac{(-1;q^6)_L}{(-1;q^2)_L}(1+q^2-q^{2(L+1)}).
\end{split}
\label{}
\end{equation}
Replace $q^2\rightarrow q$ in the above to complete the proof.
\end{proof}
\noindent
Applying Theorem~\ref{T1.1} with $a=1$ to \eqref{**} $v$-times we get
\begin{equation}
\begin{split}
&\sum_{n_1,\ldots,n_v\geq 0} \frac{q^{\sum_{i=1}^v N_i(N_i+1)+n_v}} {(q)_{n_1\ldots}(q)_{n_{v-1}}(q)_{1+2n_v}}
\frac{(-1;q^3)_{n_v}}{(-1;q)_{n_v}}\frac{(q)_{2L+1}}{(q)_{L-N_1}}(1+q-q^{1+n_v})=\\
&\sum_{j=-\infty}^\infty (\frac{j}{3}) q^{\frac{(2v+1)j^2-(2v+3)j}{2}+1} {2L+1\brack L+j}_q.
\end{split}
\label{}
\end{equation}
Letting $L\rightarrow\infty$ and using \eqref{qpi} we derive
\begin{equation}
\sum_{n_1,\ldots,n_v\geq 0} \frac{q^{\sum_{i=1}^v N_i(N_i+1)+n_v}} {(q)_{n_1\ldots}(q)_{n_{v-1}}(q)_{1+2n_v}}
\frac{(-1;q^3)_{n_v}}{(-1;q)_{n_v}}(1+q-q^{1+n_v}) = \frac{Q(q^{6v+3},-q^{2v})}{(q)_\infty}.
\label{+}
\end{equation}
It is instructive to compare \eqref{+} with the following formula
\begin{equation}
\sum_{n_1,\ldots,n_v\geq 0} \frac{q^{\sum_{i=1}^v N_i^2}} {(q)_{n_1\ldots}(q)_{n_{v-1}}(q)_{2n_v}}
\frac{(-1;q^3)_{n_v}}{(-1;q)_{n_v}} = \frac{Q(q^{6v+3},-q^{v+1})}{(q)_\infty},
\label{++}
\end{equation}
first proven by McLaughlin and Sills in [\cite{MS1}, (5.9)].
We remark that products on the right of \eqref{+} and \eqref{++} are identical when $v=1$.

Next, we will prove
\begin{theorem}\label{T1.8}{}
\begin{equation}
\sum_{j=-\infty}^\infty (-1)^j q^{j^2} \jacobi{j+1}{3} {2L+1\brack L+j}_{q^2}=\frac{(q^3;q^6)_L}{(q;q^2)_{L+1}}(1-q^{2(1+2L)}).
\label{}
\end{equation}
\end{theorem}
\begin{proof}
\noindent
Using \eqref{1.4} and \eqref{1.6} we have
\begin{equation}
\begin{split}
&\sum_{j=-\infty}^\infty (-1)^j q^{j^2} \jacobi{j+1}{3} {2L+1\brack L+j}_{q^2}=
\sum_{j=-\infty}^\infty (-1)^j\frac{w^{1+j}-\bar w^{1+j}}{w-\bar w} {2L+1\brack L+j}_{q^2} q^{j^2}=\\
&2\Re(\frac{w}{w-\bar w}\sum_{j=-\infty}^\infty q^{j^2}(-w)^j {2L+1\brack L+j}_{q^2})=
2\Re(\frac{w}{w-\bar w}(q\bar w,qw;q^2)_L(1-wq^{2L+1}))=\\
&(q\bar w,qw;q^2)_L(1+q^{2L+1})=
(1+q^{2L+1})\prod_{j=0}^{L-1}(1+q^{1+2j}+q^{2(1+2j)})=\frac{(q^3;q^6)_L}{(q;q^2)_{L+1}}(1-q^{2(1+2L)}).
\end{split}
\label{}
\end{equation}
\end{proof}
\noindent
Applying Theorem~\ref{T1.1} with $a=1$ and $q\rightarrow q^2$ $v$-times we obtain
\begin{equation}
\sum_{n_1,\ldots,n_v\geq 0} \frac{\tilde q^{\sum_{i=1}^v N_i(N_i+1)}} {(\tilde q)_{n_1\ldots}(\tilde q)_{n_{v-1}}(\tilde q)_{2n_v}}\cdot
\frac{(q^3;q^6)_{n_v}}{(q;q^2)_{1+n_v}} \cdot \frac{(\tilde q)_{2L+1}}{(\tilde q)_{L-N_1}}=
\sum_{j=-\infty}^\infty q^{(2v+1)j^2-2vj}(-1)^j {2L+1\brack L+j}_{\tilde q} \jacobi{j+1}{3},
\label{}
\end{equation}
where $\tilde q=q^2$.\\
Letting $L\rightarrow\infty$ and using \eqref{qpi} we arrive at
\begin{equation}
\sum_{n_1,\ldots,n_v\geq 0} \frac{\tilde q^{\sum_{i=1}^v N_i(N_i+1)}} {(\tilde q)_{n_1\ldots}(\tilde q)_{n_{v-1}}(\tilde q)_{2n_v}}\cdot
\frac{(q^3;q^6)_{n_v}}{(q;q^2)_{1+n_v}}= \frac{Q(q^{12v+6},q)}{(\tilde q)_\infty}.
\label{ms2}
\end{equation}
We remark that $v=1$ case of \eqref{ms2} was first proven by McLaughlin and Sills [\cite{MS2}, Thm 4.8].

Next, we will prove
\begin{theorem}\label{T1.9}{}
\begin{equation}
\sum_{j=-\infty}^\infty (-1)^j q^{j^2} \jacobi{j+1}{3} {2L\brack L+j}_{q^2}=\frac{(q^3;q^6)_L}{(q;q^2)_L}.
\label{r}
\end{equation}
\end{theorem}
\begin{proof}
\begin{equation}
\begin{split}
&\sum_{j=-\infty}^\infty (-1)^j q^{j^2} \jacobi{j+1}{3} {2L\brack L+j}_{q^2}=
\sum_{j=-\infty}^\infty (-1)^j\frac{w^{1+j}-\bar w^{1+j}}{w-\bar w} {2L\brack L+j}_{q^2} q^{j^2}=\\
&2\Re(\frac{w}{w-\bar w}(q\bar w,qw;q^2)_L=\frac{(q^3;q^6)_L}{(q;q^2)_L}.
\end{split}
\label{}
\end{equation}
\end{proof}
\noindent
Applying Theorem~\ref{T1.1} with $a=0$ and $q\rightarrow q^2$ $v$-times to \eqref{r} we obtain
\begin{equation}
\sum_{n_1,\ldots,n_v\geq 0} \frac{\tilde q^{\sum_{i=1}^v N_i^2}} {(\tilde q)_{n_1\ldots}(\tilde q)_{n_{v-1}}(\tilde q)_{2n_v}}\cdot
\frac{(q^3;q^6)_{n_v}}{(q;q^2)_{n_v}} \cdot \frac{(\tilde q)_{2L}}{(\tilde q)_{L-N_1}}=
\sum_{j=-\infty}^\infty (-1)^j \jacobi{j+1}{3}
q^{(2v+1)j^2}{2L\brack L+j}_{\tilde q},
\label{}
\end{equation}
where $\tilde q=q^2$.\\
Letting $L\rightarrow\infty$ we obtain with the aid of \eqref{qpi}
\begin{equation}
\sum_{n_1,\ldots,n_v\geq 0} \frac{\tilde q^{\sum_{i=1}^v N_i^2}} {(\tilde q)_{n_1\ldots}(\tilde q)_{n_{v-1}}(\tilde q)_{2n_v}}\cdot
\frac{(q^3;q^6)_{n_v}}{(q;q^2)_{n_v}}= \frac{Q(q^{12v+6},q^{2v+1})}{(\tilde q)_\infty}.
\label{}
\end{equation}
Case $v=1$ 
\begin{equation}
\sum_{n\geq 0} \frac{q^{2n^2}}{(\tilde q)_{2n}}\frac{(q^3;q^6)_n}{(q;q^2)_n}= \frac{Q(q^{18},q^3)}{(\tilde q)_\infty},
\label{}
\end{equation}
was first recorded by Ramanujan and proven by Andrews and Berndt in [\cite{AB}, Ent:5, 3.4].

Next, we will prove
\begin{theorem}\label{T1.10}{}
\begin{equation}
\sum_{j=-\infty}^\infty (-1)^j q^{j^2} \jacobi{j+1}{3} {2L+1\brack L-j}_{q^2}=\frac{(q^3;q^6)_L}{(q;q^2)_L}.
\label{LS(124)}
\end{equation}
\end{theorem}
\begin{proof}
\noindent
Using $q$-binomial recurrences \eqref{eq:Binom_rec} we have
\begin{equation}
\begin{split}
&\sum_{j=-\infty}^\infty (-1)^j q^{j^2} \jacobi{j+1}{3} {2L+1\brack L+1+j}_{\tilde q}=
\sum_{j=-\infty}^\infty (-1)^j q^{j^2} \jacobi{j+1}{3} {2L\brack L+j}_{\tilde q}+\\
q^{2L+1}&\sum_{j=-\infty}^\infty (-1)^j q^{j^2+2j+1} \jacobi{j+1}{3} {2L\brack L+1+j}_{\tilde q}.
\end{split}
\label{}
\end{equation}
Note that
\begin{equation}
\sum_{j=-\infty}^\infty (-1)^j q^{j^2+2j+1} \jacobi{j+1}{3} {2L\brack L+1+j}_{\tilde q}=
-\sum_{j=-\infty}^\infty (-1)^j q^{j^2} \jacobi{j}{3} {2L\brack L+j}_{\tilde q}.
\label{}
\end{equation}
Observe that the last sum on the right negates under $j\rightarrow -j$. Hence, it equals zero. \\
And so, with the aid of \eqref{r} we have
\begin{equation}
\sum_{j=-\infty}^\infty (-1)^j q^{j^2} \jacobi{j+1}{3} {2L+1\brack L-j}_{q^2}=
\sum_{j=-\infty}^\infty (-1)^j q^{j^2} \jacobi{j+1}{3} {2L\brack L-j}_{q^2}=
\frac{(q^3;q^6)_L}{(q;q^2)_L}.
\label{}
\end{equation}
\end{proof}
\noindent
Using Theorem~\ref{T1.1} with $a=1$ and $q\rightarrow q^2$ to get
\begin{equation}
\sum_{n_1,\ldots,n_v\geq 0} \frac{\tilde q^{\sum_{i=1}^v N_i(N_i+1)}} {(\tilde q)_{n_1\ldots}(\tilde q)_{n_{v-1}}(\tilde q)_{2n_v+1}}\cdot
\frac{(q^3;q^6)_{n_v}}{(q;q^2)_{n_v}}\cdot \frac{(\tilde q)_{2L+1}}{(\tilde q)_{L-N_1}}=
\sum_{j=-\infty}^\infty (-1)^j q^{2v+1)^{j^2}+2vj} \jacobi{j+1}{3} {2L+1\brack L-j}_{q^2}.
\label{}
\end{equation}
As $L\rightarrow\infty$, we obtain with the aid of \eqref{qpi}
\begin{equation}
\sum_{n_1,\ldots,n_v\geq 0} \frac{\tilde q^{\sum_{i=1}^v N_i(N_i+1)}} {(\tilde q)_{n_1\ldots}(\tilde q)_{n_{v-1}}(\tilde q)_{2n_v+1}}\cdot
\frac{(q^3;q^6)_{n_v}}{(q;q^2)_{n_v}}=\frac{Q(q^{12v+6},q^{4v+1})}{(\tilde q)_\infty},
\label{}
\end{equation}
where $\tilde q=q^2$. Case $v=1$ yields
\begin{equation}
\sum_{n\geq 0} \frac{q^{2n(n+1)}}{(\tilde q)_{2n+1}}\frac{(q^3;q^6)_n}{(q;q^2)_n}= \frac{Q(q^{18},q^5)}{(\tilde q)_\infty},
\label{}
\end{equation}
which is item (124) in Slater's compendium of Rogers--Ramanujan type identities \cite{S}.

Next, we will prove
\begin{theorem}\label{T1.11}{}
\begin{equation}
\sum_{j=-\infty}^\infty (-1)^j \jacobi{j}{3} {2L+1\brack L-j}_{\tilde q} q^{j^2}=q^{1+2L}\frac{(q^3;q^6)_L}{(q;q^2)_L}.
\label{SL(125)}
\end{equation}
\end{theorem}
\begin{proof}
\begin{equation}
\begin{split}
&\sum_{j=-\infty}^\infty (-1)^j \jacobi{j}{3} {2L+1\brack L-j}_{\tilde q} q^{j^2}=
\sum_{j=-\infty}^\infty (-1)^j\frac{w^j-\bar w^j}{w-\bar w} {2L+1\brack L+1+j}_{q^2} q^{j^2}=\\
&2\Re(\frac{1}{w-\bar w}(q\bar w;q^2)_{L+1}(qw;q^2)_L)=
2\Re(\frac{(q\bar w,qw;q^2)_L}{w-\bar w}(1-q\bar w q^{2L}))=\\
&(q\bar w,qw;q^2)_L q^{1+2L}=\frac{(q^3;q^6)_L}{(q;q^2)_L}q^{1+2L},
\end{split}
\label{}
\end{equation}
where $\tilde q=q^2$.
\end{proof}
\noindent
Applying Theorem~\ref{T1.1} with $a=1$ and $q\rightarrow q^2$ to \eqref{SL(125)} $v$-times we obtain
\begin{equation}
\sum_{n_1,\ldots,n_v\geq 0} \frac{\tilde q^{\sum_{i=1}^v N_i(N_i+1)+n_v}} {(\tilde q)_{n_1\ldots}(\tilde q)_{n_{v-1}}(\tilde q)_{2n_v+1}}
\frac{(q^3;q^6)_{n_v}}{(q;q^2)_{n_v}} \frac{(\tilde q)_{2L+1}}{(\tilde q)_{L-N_1}}=
\sum_{j=-\infty}^\infty (-1)^j q^{(2v+1)j^2+2vj-1} \jacobi{j}{3} {2L+1\brack L-j}_{q^2}.
\label{}
\end{equation}
As $L\rightarrow\infty$ we get with the aid of \eqref{qpi}
\begin{equation}
\sum_{n_1,\ldots,n_v\geq 0} \frac{\tilde q^{\sum_{i=1}^v N_i(N_i+1)+n_v}} {(\tilde q)_{n_1\ldots}(\tilde q)_{n_{v-1}}(\tilde q)_{2n_v+1}}
\frac{(q^3;q^6)_{n_v}}{(q;q^2)_{n_v}}=\frac{Q(q^{12v+6},q^{4v+3})}{(\tilde q)_\infty}.
\label{}
\end{equation}
Case $v=1$ yields item (125) on Slater's list in \cite{S}
\begin{equation}
\sum_{n\geq 0} \frac{\tilde q^{n(n+2)}}{(\tilde q)_{2n+1}}\frac{(q^3;q^6)_n}{(q;q^2)_n}= \frac{Q(q^{18},q^7)}{(\tilde q)_\infty},
\label{}
\end{equation}
where $\tilde q=q^2$.

\section{Some additional identities}\label{Sec:3}

\hskip 0.05in

In this section, we present three new polynomial identities. The proof of these identities is left as an exercise for a motivated reader.
\begin{equation}
\sum_{j=-\infty}^\infty (-1)^j \jacobi{j+1}{3} q^{j\choose 2} {2L+1\brack L-j}_q=
\left\{\begin{array}{ll}\frac{(q^3;q^3)_{L-1}}{(q)_{L-1}} (2+q^L(1+q)-q^{2L+1}), & L>0 \\ 1, & L=0\end{array}\right..
\label{y}
\end{equation}
Using Theorem~\ref{T1.1} with $a=1$ and $L=\infty$, we get with the aid of Theorem~\ref{T1.6}
\begin{equation}
1+\sum_{r\geq 1}\frac{q^{r^2+r}}{(q^2;q)_{2r}}\frac{(q^3;q^3)_{r-1}}{(q)_{r-1}} (2+q^r(1+q)-q^{1+2r})= 
\frac{(q^9,-q^2,-q^7;q^9)_\infty(q^5,q^{13};q^{18})_\infty}{(q^2;q)_\infty}.
\label{}
\end{equation}

Next,
\begin{equation}
\sum_{j=-\infty}^\infty (-1)^j \jacobi{j}{3} q^{j\choose 2} {2L+1\brack L-j}_q=
\left\{\begin{array}{ll}\frac{(q^3;q^3)_{L-1}}{(q)_{L-1}} (-1+q^L(1+q)+2q^{2L+1}), & L>0 \\ q, & L=0\end{array}\right..
\label{z}
\end{equation}
Using Theorem~\ref{T1.1} with $a=1$ and $L=\infty$, we obtain with the aid of Theorem~\ref{T1.6}
\begin{equation}
1+\sum_{r\geq 1}\frac{q^{r^2+r-1}}{(q^2;q)_{2r}} \cdot \frac{(q^3;q^3)_{r-1}}{(q)_{r-1}} \cdot (-1+q^r(1+q)+2q^{2r+1})= 
\frac{(q^9,-q^4,-q^5;q^9)_\infty(q,q^{17};q^{18})_\infty}{(q^2;q)_\infty}.
\label{}
\end{equation}

Lastly, adding \eqref{y} and \eqref{z}, and observing that 
\begin{equation*}
\jacobi{j}{3}+\jacobi{j+1}{3}+\jacobi{j+2}{3}=0,
\label{}
\end{equation*}
we get
\begin{equation}
\sum_{j=-\infty}^\infty (-1)^{j+1} \jacobi{j+2}{3} q^{j\choose 2} {2L+1\brack L-j}_q=
\left\{\begin{array}{ll}\frac{(q^3;q^3)_{L-1}}{(q)_{L-1}} (1+q^{2L+1}+2q^L(1+q)), & L>0 \\ 1+q, & L=0\end{array}\right..
\label{xy}
\end{equation}
We divide \eqref{xy} by $(1+q)$, and then use Theorem~\ref{T1.1} with $a=1$ to obtain, as $L\rightarrow\infty$ 
\begin{equation}
1+\sum_{r\geq 1}\frac{q^{r^2+r}}{(q^2;q)_{2r}} \frac{(q^3;q^3)_{r-1}}{(q)_{r-1}} (\frac{1+q^{1+2r}}{1+q}+2q^r)= 
\frac{(q^9,-q^8,-q^{10};q^9)_\infty(q^7,q^{11};q^{18})_\infty}{(q^2;q)_\infty}.
\label{w}
\end{equation}

\section*{Acknowledgements}

The author would like to thank George E. Andrews, Aritram Dhar, and Ali K. Uncu for their kind interest.


\begin{thebibliography}{99}

\bibitem{A1}  G. E. Andrews, \textit{The theory of partitions}, Cambridge Mathematical Library, Cambridge University Press, Cambridge, 1998 
              Reprint of the 1976 original. MR1634067 (99c:11126).

\bibitem{A}   G. E. Andrews, \textit{Multiple series Rogers--Ramanujan type identities}, Pac. J. Math. \textbf{114} (1984), 267--283.

\bibitem{AB}  G. E. Andrews, B. C. Berndt, \textit{Ramanujan's lost notebook. Part II}, Springer, New York, 2009.

\bibitem{BU}  A. Berkovich, A. K. Uncu, \textit{New infinite hierarchies of polynomial identities related to the Capparelli partition theorems}, 
              J. Math. Anal. Appl. \textbf{506} (2022), 125678. 

\bibitem{C_thesis} S. Capparelli, \textit{Vertex operator relations for affine algebras and combinatorial identities}, Ph.D Thesis Rutgers University, 1988.
						
\bibitem{C}   S. Capparelli, \textit{A combinatorial proof of a partition identity related to the level 3 representation of twisted affine Lie algebra},
              Communications in Algebra \textbf{23} (1995), no. 8, 2959--2969.

\bibitem{GR}  G. Gasper, M. Rahman, \textit{Basic hypergeometric series}, Cambridge University Press, 2009.

\bibitem{KR}  S. Kanade and M. Russell, \textit{Staircases to analytic sum-sides for many new integer partition identities of
              Rogers--Ramanujan type}, Electron. J. Combin. \textbf{26} (2019), no. 1, Paper 1.6.

\bibitem{K}   K. Kur\c{s}ung\"oz, \textit{Andrews--Gordon type series for Capparelli's and G\"ollnitz--Gordon identities}, 
              J. Combin. Theory Ser. A \textbf{165} (2019), 117--138.

\bibitem{MS1} J. McLaughlin, A. V. Sills, \textit{Ramanujan--Slater type identities related to the moduli $18$ and $24$}, 
              J. Math. Anal. Appl. \textbf{344} (2008), 765--777.
							
\bibitem{MS2} J. McLaughlin, A. V. Sills, \textit{Combinatorics of Ramanujan--Slater type identities}, 
              Comb. Numb. Theory, \textbf{}Walter de Gruter, Berlin (2009), 125--139.					
												
\bibitem{P}   P. Paule, \textit{Zwei neue Transformationen als elementare Anwendungen derq--Vandermonde Formel}, Ph.D. Thesis, University of Vienna, 1982.

\bibitem{S}   L. J. Slater, \textit{Further identities of the Rogers--Ramanujan type},
              Proc. London Math. Soc. (2) \textbf{54} (1952), 147--167.
			
\end{thebibliography}
\end{document}